\documentclass[12pt]{article}
\usepackage{caption}
\usepackage{url}
\usepackage{color}
\usepackage{amssymb}
\usepackage{amsmath}
\usepackage{amsthm}
\usepackage{graphicx}
\usepackage{mathscinet}

\newtheorem{theorem}{Theorem}[section]
\newtheorem{lemma}[theorem]{Lemma}

\newtheorem{cor}[theorem]{Corollary}
\newtheorem{quest}[theorem]{Question}
\newtheorem{prop}[theorem]{Proposition}

\title{Remarks on the Spectral Approach to Finding Short Paths}
\author{Kelly Yancey \thanks{Institute for Defense Analyses / Center for Computing Sciences (IDA / CCS), kbyance@super.org} and Matthew P. Yancey \thanks{Institute for Defense Analyses / Center for Computing Sciences (IDA / CCS), mpyance@super.org}}

\begin{document}
\maketitle

\begin{abstract}
We are interested in the general question: to what extent are the spectral properties of a graph connected to the distance function?
Our motivation is a concrete example of this question that is due to Steinerberger.
We provide some negative results via constructions of families of graphs where the spectral properties disagree with the distance function in a manner that answers many of the questions posed by Steinerberger.
We also provide a positive result by replacing Steinerberger's set-up with conditions involving graph curvature.
\end{abstract}

\section{Introduction}

We assume connected graph $G$ with vertex set $V$ and edge set $E$ is fixed.
We will use the notation $V = \{1,2,\ldots , n\}$.
The Laplacian of $G$ is $L = D - A$, where $A$ is the adjacency matrix and $D$ is the diagonal degree matrix.
We assume some vertex $i$ is fixed, which we will call \emph{the special vertex}.
Let $L_i$ denote the minor of $L$ formed by removing row $i$ and column $i$.
The length of a path is the number of edges it contains.

The eigenvector associated to the second smallest eigenvalue of $L$ is called the Fiedler vector.
When appropriately normalized, the Fiedler vector---interpreted as a function $f:V \rightarrow \mathbb{R}$---is known to be a solution to the optimization problem
\begin{equation}\label{original laplace as optimization}
 \min_{\|g\|_2 = 1, 0 = \sum_j g(j)} \sum_{jk \in E} |g(k) - g(j)|^2 . 
\end{equation}
We define $V_i = V \setminus \{i\}$, and let $f_i:V_i \rightarrow \mathbb{R}$ be the eigenvector associated to the smallest eigenvalue of $L_i$.
Steinerberger \cite{S} recently posed many questions about the relationship between the values of $f_i$ and the set shortest paths in $G$ with endpoint $i$.
In this connection, the function $f_i$ is lifted to have domain $V$ by setting $f_i(i) = 0$.
This lift allows us to describe $f_i$ as the solution to an optimization problem that resembles the characterization of the Fiedler vector of the Laplacian: $f_i$ is a solution to 
\begin{equation}\label{f_i as optimization problem}
 \min_{\|g\|_2 = 1, 0 = g(i)} \sum_{jk \in E} |g(k) - g(j)|^2 . 
\end{equation}
To facilitate the connection between $f_i$ and the shortest paths beginning at $i$, Steinerberger proved several results about $f_i$.

Let $G-i$ denote the graph $G$ after vertex $i$ is deleted.
If $G-i$ is not connected, then the problem of finding paths in $G$ can be separated into independent sub-problems.

\begin{theorem}[\cite{S}]\label{core result}
If $G-i$ is connected, then the following are true about $f_i$.
\begin{itemize}
	\item If $f_i(u) > 0$ for some vertex $u$, then $f_i(u') > 0$ for all $u' \in V_i$.  By multiplying all values of $f_i$ by $-1$, we will assume for the rest of the paper that $f_i$ is positive over $V_i$.
	\item For $u \in V_i$, there exists a $u' \in N(u)$ (possibly $u' = i$) such that $f_i(u') < f_i(u)$.
\end{itemize}
\end{theorem}

By the above results, for each $i$ one can construct a spanning tree $T_i$ of $G$ by connecting each vertex of $V_i$ to the neighbor that minimizes $f_i$.
We will call the unique path in $T_i$ from $j$ to $i$ the \emph{spectral path from $j$ to $i$}, which is not necessarily the same collection of edges and vertices as the spectral path from $i$ to $j$.
Steinerberger \cite{S} gave ample empirical evidence that a spectral path is commonly also a shortest path between two vertices, and frequently has length that exceeds the distance between the endpoints by a small bound.
The \emph{symmetric spectral path between $i$ and $j$} is the shorter of the spectral path from $i$ to $j$ and the spectral path from $j$ to $i$ (if they are the same length, chose among them arbitrarily).

Using a spectral path to find a short (if not shortest) path between vertices in a graph was motivated by Steinerberger as a discretization of a process in partial differential equations (PDE) about which several important questions, like the ``Hot Spots conjecture,'' remain open.
The hope is that a better understanding of spectral paths in graphs might possibly translate into new results in PDE.
Steinerberger mentioned many different ideas that deserved exploration---in \cite{S} there are six open problems in Section 2.6 and two more in Section 3.3.  
The main result of this paper is to resolve a majority of these questions, which we summarize as a single question next.

\begin{quest}[\cite{S}]\label{motivating question}
Are there graphs such that, for some pair of vertices $i \neq j$, the spectral path from $i$ to $j$ is $c$-times the distance between $i$ and $j$ for any $c >1$?  
This question can be weakened in four senses.  
First, what if we restrict to planar graphs?  
Second, what if vertices $i$ and $j$ are chosen randomly (in other words, does the bound hold for all but an asymptotically-small subset of pairs $\{i,j\}$)?  
Third, can we at least bound the length of the symmetric spectral path? Finally, can the length of a spectral path be bounded as a multiple of the diameter of the space?
\end{quest}

We will give a negative answer to Question \ref{motivating question} (including all four relaxations) by constructing families of graphs where spectral paths are unboundedly long between pairs of vertices at bounded distance apart.

We also have a positive result, which partially answer Question 6 of Section 2.6 in \cite{S}, which asks for ``effective variants'' of Steinerberger's 
set-up.
Our approach to achieving a positive result is to incorporate the field of \emph{graph curvature}, which is the study of discretizations of various algebraic and geometric definitions of a curved space.
Our result connects the distance between two vertices with the length of a path formed by locally-optimal choices over a spectral function, although we use a different spectral formulation than (\ref{f_i as optimization problem}).
In particular, we replace the Fiedler vector of the Laplacian with the ``spread'' of a graph, which we will define in Section \ref{pos result sec} (its connection to the Fiedler vector of a Laplacian is described in Section 1.3 of \cite{Y_Pos}).

Of the eight questions stated in \cite{S}, this collection of results leaves only one fully open: Question 4 of Section 2.6 in \cite{S} asks about the computational efficiency of spectral paths.

We will construct the family of graphs that answer Question \ref{motivating question} in Section \ref{neg results sec}.
We will prove our positive result and ask a question that could lead to another positive result in Section \ref{pos result sec}.

\section{Constructions}\label{neg results sec}

The empirical results in \cite{S} showed that the spectral paths were short on random graphs; the spectral paths had small inefficiencies on ``deterministic graphs.''
Our families of graphs will be highly structured.
Our approach to analyzing the affect of these structures on spectral paths will be common among the different graph families we consider.
Thus, the core arguments that will be repeated are combined and presented in Section \ref{weights subsec}.
Although only a subsection, Section \ref{weights subsec} is by far the largest part of this paper and with the most sophisticated arguments.
We construct three graph families, which are presented in each of Sections \ref{const 1 subsec}, \ref{const 2 subsec}, and \ref{const 3 subsec}, respectively.

\subsection{Using Symmetries to Simulate Weights}\label{weights subsec}
A graph automorphism is a bijection $\pi:V \rightarrow V$ such that $xy \in E$ if and only if $\pi(x)\pi(y) \in E$.

\begin{theorem}\label{core tool}
The eigenvector $f_i$ is the unique eigenvector of $L_i$ where each coordinate is nonnegative.  This implies that for any graph automorphism $\pi$ that is fixed on $i$, we have that $f_i$ is invariant under $\pi$.
\end{theorem}
\begin{proof}
It is folklore that because $L$ is symmetric and real (and therefore Hermitian), its eigenvectors are real-valued and pairwise orthogonal.
The assumptions of the previous sentence also apply to $L_i$, and therefore so does the conclusions.
By Theorem \ref{core result} $f_i$ is positive over all coordinates in $V_i$.
Any nonzero real vector that is orthogonal to $f_i$ must have a negative coordinate, implying that $f_i$ is the unique nonnegative eigenvector.
If $\pi$ is a graph automorphism that fixes $i$, then $f_i \circ \pi$ is a positive eigenvector of $L_i$ (with the same eigenvalue as $f_i$).
By the uniqueness of $f_i$, this implies $f_i \circ \pi = f_i$.
\end{proof}

In this subsection, we will show how to generalize the current set up to account for both vertex weights and edge weights.
The edge weights take the form of a symmetric function $w_E: V \times V \rightarrow \mathbb{R}_{\geq 0}$ that satisfies $w_E(u,v) = 0$ if $uv \notin E$.
The vertex weights take the form of a function $w_V:V \rightarrow \mathbb{R}_{\geq 0}$ with at least one positive value.
For a function $g:V \rightarrow \mathbb{R}$, we define $\|g\|_w := \sqrt{\sum_{u \in V} w_V(u)g(u)^2}$.
In the weighted version, the function $f_i$ is a function that solves the optimization problem 
\begin{equation}\label{f_i as weighted optimization problem}
 \min_{\|g\|_w = 1, 0 = g(i)} \sum_{jk \in E} w_E(j,k) |g(k) - g(j)|^2 . 
\end{equation}
It should be pointed out that we have not proven that $f_i$ is unique (although we will prove this under modest assumptions in Theorem \ref{uniqueness of solutions}) so $f_i$ is to be treated as some arbitrary solution.

Observe that in the above statement, the original unweighted version can be recovered when the weight of each edge and vertex is set to $1$.
Importantly, we still consider the length of a path to be the number of edges it contains.
Because this choice is atypical of the field, let us emphasize that the length of a path is independent of the weight functions.
When computing $f_i$, we can assume that all edges have positive weight, as deleting them will not affect (\ref{f_i as weighted optimization problem}).
However, they still play a role in determining the distance between vertices.

Question \ref{motivating question} is not interesting for weighted graphs, as any induced path in any graph can be made a symmetric spectral path by giving the corresponding edges sufficiently large weight.
We introduce weights because the spectral paths of some large unweighted graphs can be determined by reducing the problem to smaller graphs with weights.

Let $u \sim_i v$ denote the relation between vertices in $u,v \in V_i$ that there exists a graph automorphism $\pi$ that fixes $i$ and $\pi(u) = v$.
Because graph automorphisms are closed under inverses and compositions, $\sim_i$ forms an equivalence relation.
We construct weighted graph $G_i$ from $G$ by creating a bijection between the vertices of $G_i$ and the $\sim_i$-equivalence classes of $G$.
Let $\phi_i :V(G) \rightarrow V(G_i)$ be the mapping such that $\phi_i$ maps each vertex to the equivalence class containing that vertex.
The weight functions of $G_i$ are defined as follows: $w_V(u) = |\phi_i^{-1}(u)|$ and $w_E(uv)$ is the number of edges with an endpoint in $\phi^{-1}(u)$ and the other endpoint in $\phi^{-1}(v)$.
Let $\widehat{f}_i$ be a solution to the optimization problem (\ref{f_i as weighted optimization problem}) over $G_i$ with special vertex $\phi_i(i)$.
We omit the proof to the following theorem, as it clearly follows from Theorem \ref{core tool} and the construction of $G_i$.

\begin{theorem}\label{reduction theorem}
If each $\sim_i$-equivalence class of $G$ contains no edges, then we can construct\footnote{We allow the non-uniqueness of $f_i$ to correspond with the possibly non-uniqueness of $\widehat{f}_i$.} $f_i$ from $\widehat{f}_i$ by setting $f_i(j) = \widehat{f}_i(\phi_i(j))$ for each $j \in V$.
\end{theorem}

We now proceed to study the function $f_i$ in this weighted context; for the rest of the subsection we assume $G$ is a weighted graph.
The function $w_E$ is represented in matrix form as $W_E$ and the function $w_V$ is represented as diagonal matrix $W_D$.
The degree of a vertex $k$ is defined to be $d(k) := \sum_{j \in V}w_E(kj)$.
The Laplacian for a graph with edge weights is $L = D - W_E$.

We define a function $c$ when $\|g\|_w >0$ as 
\begin{equation}\label{defining c}
 c(w_V, w_E, g) = \frac{\sum_{jk \in E} w_E(k,j) |g(k) - g(j)|^2}{\|g\|_w^2}.
\end{equation}
A function that solves optimization problem (\ref{f_i as weighted optimization problem}) for fixed $w_V, w_E$ is characterized as the function $g/\|g\|_w$ where $g$ minimizes $c(w_V, w_E, g)$ under the criteria $g(i) = 0$.
Let $g_i \in \mathbb{R}^{n-1}$ be the vector representation of $g$ restricted to $V_i$, and let $W_{V,i}$ be $W_V$ with row $i$ and column $i$ deleted.
Observe that when $g(i) = 0$ we have
$$c(w_V, w_E, g) = \frac{g_i^T L_i g_i}{g_i^T W_{V,i} g_i}. $$
We say that a graph is ``positively connected'' if there exists a path between any two vertices using only edges with positive weights.

\begin{prop}
If $G$ is positively connected, then $L_i$ is positive definite.
\end{prop}
\begin{proof}
If $M$ is symmetric and real, then for $M$ to be positive definite it is necessary and sufficient that for any nonzero real vector $x$ we have $x^T M x > 0$.
Let us interpret $x$ as a function from $V_i$ to $\mathbb{R}$ instead of a vector $\mathbb{R}^{V_i}$.
Next, we lift $x$ to $x_*$ which domain $V$ by setting $x_*(i) = 0$.
Under this setting, we have $x^T L_i x = \sum_{jk \in E} w_E(j,k) (x_*(k) - x_*(j))^2$.
Because $x$ is a nonzero vector, there is some index $j$ such that $x(j) \neq 0$.
By the assumption, there exists a path $i=u_1,u_2, \ldots, u_k=j$ using only positively weighted edges.
Let $t$ be the smallest index such that $x(u_t) \neq x(u_{t+1})$.
Therefore 
\begin{eqnarray*}
x^T L_i x &=& \sum_{jk \in E} w_E(j,k) (x_*(k) - x_*(j))^2\\
	& \geq & w_E(u_t, u_{t+1})(x_*(u_t) - x_*(u_{t+1}))^2\\
	&  > & 0.
\end{eqnarray*}
\end{proof}

If $L_i$ is positive definite, then because it is also real and symmetric, it has a Choleskey decomposition $L_i = U_i^T U_i$, where $U_i$ is real and invertible.
For $g_i \in \mathbb{R}^{V_i}$, let $g_i' = U_i g_i$ so that 
$$c(w_V, w_E, g) = \frac{\|g_i'\|_2^2}{g_i'^T (U_i^T)^{-1}W_{V,i} U_i^{-1} g_i'}.$$
Let $L_i' = (U_i^{-1})^T W_{V,i} U_i^{-1}$, which is real and symmetric.
Applying the Courant-Fischer theorem, $c(w_V, w_E, g)$ is minimized by choosing $g_i'$ to be the dominant eigenvector of $L_i'$.
Reverting the change of variables gives the following statement.

\begin{lemma} \label{matrix representation}
When $G$ is positively connected, the solutions that optimize (\ref{f_i as weighted optimization problem}) correspond with the dominant right-eigenvectors of $L_i^{-1}W_{V,i}$.
\end{lemma}

With this set up, we are able to generalize Theorem \ref{core result} to weighted graphs under modest assumptions.
For ease of presentation, we split it into two statements.

\begin{theorem}\label{core weighted result part 1}
Assume $G-i$ is positively connected.
Let $f_i$ be a function that optimizes (\ref{f_i as weighted optimization problem}).
If $f_i(u) > 0$ for some vertex $u$, then $f_i(u') > 0$ for all $u' \in V_i$.  
\end{theorem}
\begin{proof}
We prove the first statement by contradiction.
We begin by assuming that $f_i$ is nonnegative.
Let $Z = \{j : f_i(j) = 0, j \neq i\}$ and assume that $Z \neq \emptyset$.  
Because $G-i$ is positively connected, there exists an edge with positive weight with one endpoint in $Z$ and the other endpoint in $V_i \setminus Z$.
Suppose we modify $f_i$ by increasing $f_i(j)$ for each $j \in Z$ by some $\epsilon>0$, and consider the affect that change would have on $c(w_V, w_E, f_i)$. 
Specifically, consider the definition of $c$ in (\ref{defining c}): the denominator would increase, and when $\epsilon$ is sufficiently small the numerator would decrease.
This contradicts the choice of $f_i$.

So now suppose the sets $N =  \{j : f_i(j) < 0, j \neq i\}$ and $P =  \{j : f_i(j) > 0, j \neq i\}$ are each nonempty.
We define function $f_i^+ = |f_i|$.
Observe that $\|f_i^+\|_w = \|f_i\|_w$.
If there is an edge with positive weight with an endpoint in $N$ and another endpoint in $P$, then $c(w_V, w_E, f_i^+) <  c(w_V, w_E, f_i)$, contradicting the choice of $f_i$.
Therefore, because $G-i$ is positively connected, there are edges with positive weights connecting $P$ to $Z$ and $Z$ to $N$, and thus $N \neq \emptyset$.
Moreover, we have $c(w_V, w_E, f_i^+) =  c(w_V, w_E, f_i)$, and $f_i^+$ is a second solution to the optimization problem (\ref{f_i as weighted optimization problem}).
But this is a contradiction, because $f_i^+$ is not a solution to the optimization problem (\ref{f_i as weighted optimization problem}), as $f_i^+$ exactly matches the case described in the previous paragraph.
\end{proof}

\begin{theorem}\label{core weighted result part 2}
Assume $G$ and $G-i$ are each positively connected.
Let $f_i$ be a function that optimizes (\ref{f_i as weighted optimization problem}).
For  $u \in V_i$, there exists a $u' \in N(u)$ (possibly $u' = i$) such that $f_i(u') \leq f_i(u)$.
Moreover, we may assume the inequality is strict if $w_V(u) > 0$.
\end{theorem}
\begin{proof}
Let $g_i \in \mathbb{R}^{V_i}$ be the vector representation of $f_i$.
Recall that $g_i$ is a $\lambda$-eigenvector of $L_i^{-1}W_{V,i}$, where $\lambda$ is the largest eigenvalue.
Moreover, $L_i^{-1}W_{V,i}$ is similar to $L_i'$.
For any vector $x \in \mathbb{R}^{V_i}$ we have $x^T L_i' x = \| U_i ^{-1} x \|_w$, and therefore $L_i'$ is positive semi-definite (positive definite if and only if each vertex has positive weight).
This implies $\lambda \geq 0$, and because $\lambda$ is the largest eigenvalue we have that $\lambda > 0$ unless $W_{V,i} = 0$.
But then the weight function $w_V$ is identically $0$, which contradicts our definition of a vertex weight function.

We now consider the equation $W_{V,i} g_i = \lambda L_i g_i$.
Applying this formula at vertex $u$ with the definition of $L_i$ we get that
\begin{eqnarray*}
0	& \leq & w_V(u) g_i(u) \\
	& = & \lambda \sum_{j \neq u} w_E(u,j)( g_i(u) - g_i(j) ) \\
	& \leq & \lambda d(u) \max_j ( g_i(u) - g_i(j) ). \\
\end{eqnarray*}
If $G$ is positively connected, then $d(u) > 0$.
Thus, $\lambda d(u) > 0$, and we conclude that $\max_j ( g_i(u) - g_i(j) ) \geq 0$.
Moreover, if $w_V(u) > 0$, then by Theorem \ref{core weighted result part 1} we have a strict inequality in the first line of the equation array above, which gives the stronger conclusion $\max_j ( g_i(u) - g_i(j) ) > 0$.
\end{proof}

Unfortunately, Theorem \ref{core weighted result part 2} can not be strengthened.
It is a simple calculus exercise directly applied to the formulation (\ref{f_i as weighted optimization problem}) that if $w_V(u) = 0$, then the value $f_i(u)$ is the average of the value of its neighbors (formally, $f_i(u) = d(u)^{-1}\sum_{j \neq u}w_E(u,j) f_i(j)$).
So, if $u$ is incident with merely a single edge $ux$, then $f_i(u) = f_i(x)$.  
On the other hand, it also follows that if $w_V(u) = 0$ and $u$ has a neighbor $x$ such that $f_i(x) > f_i(u)$, then $u$ has another neighbor $y$ such that $f_i(y) < f_i(u)$.

\begin{theorem}\label{spectral paths are well-defined}
The spectral path from $j$ to $i$ is well-defined if $G$ and $G-i$ are each positively connected and $w_V(j) > 0$.
\end{theorem}
\begin{proof}
Let us attempt to construct the tree $T_i$ as described in the introduction, only vertex $x$ has no arcs leaving it if $f_i(x) \geq \max_{yx \in E} f_i(y)$.
In other words, we only include an arc in $T_i$ if the corresponding value in $f_i$ strictly decreases along that arc.

Let us consider the maximal directed path in $T_i$ starting at $j$.
That is, we consider the sequence $j=j_0,j_1,j_2,\ldots,j_k$ where for $\ell < k$ we have that $j_{\ell+1}$ is the neighbor of $j_\ell$ that minimizes $f_i(j_{\ell+1})$, and $k$ is the first index that satisfies (A) $j_k = i$, (B) every neighbor $x$ of $j_k$ satisfies $f_i(x) \geq f_i(j_k)$, or (C) the vertex $y$ that minimizes $f_i(y)$ among neighbors of $j_k$ satisfies $y \in \{j_0, j_1, \ldots, j_{k-1}\}$. 
The spectral path from $j$ to $i$ is well-defined by definition when we are in case (A) above.

By way of contradiction, we assume that we are in case (B) or (C) above.
The sequence $f_i(j_0), f_i(j_1), \ldots, f_i(j_k)$ is monotone decreasing, which implies that if we are in case (C), then we are also in case (B).
So assume that every neighbor $x$ of $j_k$ satisfies $f_i(x) \geq f_i(j_k)$.
By Theorem \ref{core weighted result part 2}, this implies $w_V(j_k) = 0$.
By the assumption $w_V(j) > 0$, this implies $k \geq 1$.
Observe that because $k$ is chosen as the first index where (B) occurs, we know that $f_i(j_{k-1}) > f_i(j_{k})$ and $j_{k-1}$ is a neighbor of $j_k$.
By the discussion prior to the statement of the theorem, this implies that we can not be in case (B).
This contradicts our assumption, so we must be in case (A), and the theorem is proven.
\end{proof}

Finding a function that solves the optimization problem (\ref{f_i as weighted optimization problem}) for a fixed graph is an easy problem computationally.
But proving that a constructed function is a solution for a family of graphs is quite hard, even for unweighted graphs.
Still, we can have some control over spectral paths using the following theorem.

\begin{theorem}\label{tool to construct paths}
Fix three vertices $j,i,k$.
Suppose $G-j$ is not connected, with $i$ and $k$ in different components of $G-j$.
Under these assumptions, if $G$ and $G-i$ are each positively connected and $w_V(k) > 0$, then $f_i(k) > f_i(j)$.
Moreover, the spectral path from $k$ to $i$ will pass through $j$.
\end{theorem}
\begin{proof}
By Theorem \ref{spectral paths are well-defined}, there is a spectral path from $k$ to $i$.
By definition, $f_i$ monotonically decreases along this path.
The assumption is equivalent to stating that every path from $k$ to $i$ includes $j$, and therefore $j$ is in the spectral path from $k$ to $i$.
\end{proof}

Let us remark that a spectral path may not be completely deterministic if a vertex had multiple neighbors with the smallest value under $f_i$, but Theorem \ref{tool to construct paths} still holds regardless of how such ties are resolved.
This implies that such ties can not occur in certain circumstances.

\begin{cor}\label{unique spectral path}
Fix three vertices $j,i,k$.
Suppose $G-j$ is not connected, with $i$ and $k$ in different components of $G-j$.
Moreover, assume $jk \in E$.
Under these assumptions, if $G$ and $G-i$ are each positively connected and $w_V(k) > 0$, then $j$ is the unique vertex that satisfies $f_i(j) = \min_{z \in N(k)} f_i(z)$.
\end{cor}
\begin{proof}
By Theorem \ref{tool to construct paths}, there is a spectral path from $k$ to $i$ and it passes through $j$.
By way of contradiction, assume that there exists a $z \in N(j)$ with $f_i(z) \leq f_i(j)$.
Then there is a spectral path from $k$ to $i$ that uses edge $kz$.
But such a spectral path will eventually pass through $j$, which contradicts that $f_i$ strictly decreases along spectral paths.
\end{proof}

Our presentation in the next subsections of the counterexamples to Question \ref{motivating question} will use the following outline.
We will construct a sequence of graphs, where each member of that sequence reduces to the same fixed graph $G_i$ under the mapping defined by Theorem \ref{reduction theorem}---but with different weights.
In our sequence of graphs, some $\sim_i$-equivalence classes will remain fixed in size, while others will grow asymptotically large.
We will scale $w_V$ and $w_E$ so that the largest weights remain fixed in value, and the smallest weights become infinitesimally small.
We will consider the limit of this sequence; removing edges with weight zero will allow us to apply Theorem \ref{tool to construct paths}.
The penultimate step of the argument is stating some graph in the sequence of graphs will be close enough to the limit of the sequence that it will have the same spectral paths.
Finally, Theorem \ref{spectral paths are well-defined} will imply that a spectral path exists, Theorem \ref{tool to construct paths} will describe what that path looks like, and Theorem \ref{reduction theorem} will provide a ``lift'' of the spectral path in the weighted graph to a (family of) spectral paths in the original unweighted graph.

The rest of this subsection will be spent on a key part of the above argument that still needs to be shown: for the penultimate step, we require that the spectral paths in the graphs in the sequence are somehow related to the spectral paths in the limiting graph.
That is, we need to show that $f_i$ is continuous under perturbations of $w_V$ and $w_E$.
We will do this by recalling that $f_i$ is the eigenvector of some matrix, which are known to be continuous under perturbations of that matrix when the eigenvalue is simple.
This is is a bit of a surprising result in its own right: eigenvectors are known to be discontinuous, and spectral graph theory is filled with examples of eigenvalues with large multiplicity (for example, see \cite{AvDF} and the references therein).

\begin{theorem}\label{uniqueness of solutions}
If $G$ and $G-i$ are each positively connected, then the largest eigenvalue of $L_i^{-1}W_{V,i}$ is simple.  Equivalently (up to reflection from positive to negative values) the function that solves the optimization problem (\ref{f_i as weighted optimization problem}) is unique.
\end{theorem}
\begin{proof}
The equivalence between the two statements is given by Lemma \ref{matrix representation}.
Let $\lambda$ be the largest eigenvalue of $L_i^{-1}W_{V,i}$.

By way of contradiction, assume $\lambda$ has multiplicity at least two.
Recall that $L_i^{-1}W_{V,i}$ is similar to $L_i$, which is real and symmetric, and therefore diagonalizable.
Therefore the $\lambda$-eigenspace of $L_i^{-1}W_{V,i}$ has dimension at least two.
Let $f_{i,1}$ and $f_{i,2}$ be independent $\lambda$-eigenvectors of $L_i^{-1}W_{V,i}$.

Because $\lambda$-eigenvectors is closed under linear combinations, there exists a vector $f_{i,*} = \alpha_1 f_{i,1} + \alpha_2 f_{i,2}$ for $\alpha_1, \alpha_2 \in \mathbb{R}$ that is a $\lambda$-eigenvector of $L_i^{-1}W_{V,i}$ and has a negative and a positive coordinate.
By Lemma \ref{matrix representation}, $f_{i,*}$ is a solution to the optimization problem (\ref{f_i as weighted optimization problem}).
But this contradicts Theorem \ref{core weighted result part 1}.
\end{proof}

\begin{theorem}\label{continuous function}
The function $f_i$ that solves the optimization problem (\ref{f_i as weighted optimization problem}) is continuous under perturbations of $w_V$ and $w_E$ inside the space of weights where $G$ and $G-i$ are positively connected.
\end{theorem}
\begin{proof}
Assume that the vertex set and edge set of our graph $G$ is fixed, as well as the special vertex $i$.
Let $\mathcal{W}_V$ denote the space of vertex weights and let $\mathcal{W}_E$ denote the space of edge weights such that $G$ and $G-i$ are positively connected.
Let $M_{n-1}(\mathbb{R})$ be the space of $(n-1)\times(n-1)$ real matrices.
We define a function $\tau : \mathcal{W}_V \times \mathcal{W}_E \rightarrow M_{n-1}(\mathbb{R})$ as $\tau(w_V, w_E) = L_i^{-1}W_{V,i}$.

We can think of elements of $\mathcal{W}_V \times \mathcal{W}_E$ as points in $\mathbb{R}^{V + E}$ and an element of $M_{n-1}(\mathbb{R})$ as a point in $\mathbb{R}^{(n-1)^2}$.
This frame of context allows us to apply natural norms to these spaces (such as $\ell_2$), and we claim that $\tau$ is continuous in such a setting.
As $L_i$ is independent of $w_V$, $\tau$ is clearly continuous under perturbation of the $w_V$ variable.
Observe that $W_{V,i}$ is independent of $w_E$ and $L_i$ is continuous under perturbations of $w_E$.
Because matrix inverse is a continuous operation, $\tau$ is continuous under perturbation of the $w_E$ variable.
This proves the claim.

We define a function $\rho : M_{n-1}(\mathbb{R}) \rightarrow \mathbb{R}^{n-1}$ that takes a matrix and returns an (arbitrarily chosen) eigenvector corresponding to the largest eigenvalue $\lambda$.
It is folklore\footnote{See, for example,\\ \url{https://mathoverflow.net/questions/207452/} or\\ \url{https://math.stackexchange.com/questions/1133071/} or\\ \url{https://math.stackexchange.com/questions/807144/} for different approaches to this.} that when $\lambda$ is a simple eigenvalue, $\rho$ is continuous.
By Theorem \ref{uniqueness of solutions}, $\rho$ is continuous.

As constructed, Lemma \ref{matrix representation} implies that $f_i = \rho(\tau(w_V, w_E))$.
The proof concludes by stating that the composition of two continuous functions is continuous.
\end{proof}

\begin{cor}\label{limiting works one step}
Suppose that $G$ and $G-i$ are positively connected.
Also, suppose for fixed vertices $j,i,k$, we have $G-j$ is not connected, with $i$ and $k$ in different components of $G-j$, and $jk \in E$.
There exists an $\epsilon > 0$ such that if $w_E$ and $w_V$ are perturbed by at most $\epsilon$ where the set of neighbors of $k$ remains fixed (we allow the creation/deletion of an edge with weight $\epsilon$ that is not incident with $k$), then the spectral path from $k$ to $i$ will still include $j$.
\end{cor}
\begin{proof}
This follows from the continuity from Theorem \ref{continuous function} and the uniqueness of the spectral path from Corollary \ref{unique spectral path}.
\end{proof}

By construction of spectral paths, if $j$ is on the spectral path $P$ from $k$ to $i$, then the spectral path from $j$ to $i$ is a sub-path of $P$.
We can thus iteratively apply Corollary \ref{limiting works one step} to achieve the next statement.

\begin{cor}\label{limiting works}
Suppose that $G$ and $G-i$ are positively connected and that there exists a unique path $P$ from vertex $k$ to $i$.
There exists an $\epsilon > 0$ such that if $w_E$ and $w_V$ are perturbed by at most $\epsilon$ where the set of neighbors of vertices in $P$ remains fixed (we allow the creation of an edge with weight $\epsilon$ that may create a new path from $k$ to $i$), then the spectral path from $k$ to $i$ will still be $P$.
\end{cor}

\subsection{Weighted Cycle}\label{const 1 subsec}

We construct an unweighted graph family $G_{\ell,k}$, which is indexed by positive integers $\ell,k$.
For fixed $\ell,k$, we construct $G_{\ell,k}$ as follows.
Start by creating disjoint paths $x_{0,i}x_{1,i}x_{2,i}\ldots x_{\ell,i}$ for each $1 \leq i \leq k$.
The vertices $x_{0,1},x_{0,2},\ldots ,x_{0,k}$ are identified to a single vertex $u$, and the vertices $x_{\ell,1},x_{\ell,2},\ldots ,x_{\ell,k}$ are identified to a single vertex $v$.
Finally, we add edge $uv$.
Figure \ref{construction 1 image} illustrates the outcome of this construction when $\ell = 5$ and $k=2$.
Observe that $G_{\ell,k}$ is a planar graph for all $\ell$ and $k$.
Let us consider the spectral paths for special vertex $u$.

\begin{figure}[htbp]
\begin{center}
\includegraphics[width=2in]{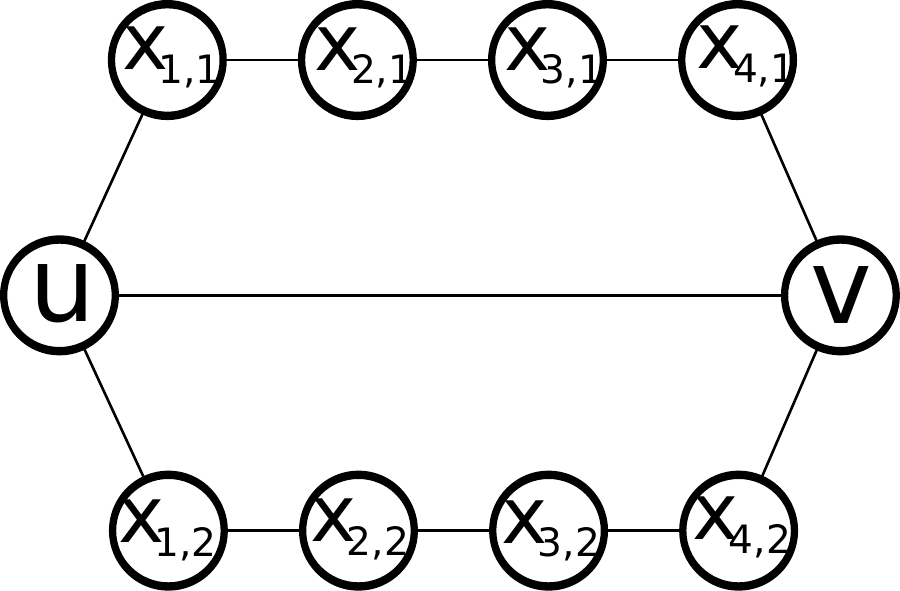}
\caption{This is graph $G_{5,2}$.}\label{construction 1 image}
\end{center}
\end{figure}

\begin{prop}\label{describing equivalence classes is weighted cycle}
Consider $G_{\ell,k}$ for $k \geq 2$.
The equivalence classes of $\sim_u$ are $\{x_{1,i}\}_i, \{x_{2,i}\}_i, \ldots, \{x_{\ell-1,i}\}_i,\{v\}$. 
\end{prop}
\begin{proof}
The map $\pi$ that fixes $u$ and $v$ and $\pi(x_{j,i}) = \pi(x_{j,i+1})$ is a graph automorphism.
Therefore $x_{j,i} \sim_u x_{j',i'}$ if $j = j'$.
This proves that each of the desired sets is a subset of an equivalence class; the remainder of the proof will be to show that no other equivalence relations exist.
To do so, we repurpose the the idea of ``landmarks'' in a graph.
We make use of two facts about a graph automorphism $\pi$: that $d(a) = d(\pi(a))$ and $d(a,b) = d(\pi(a),\pi(b))$ for any pair of vertices $a,b$.
In particular, if $\pi(b) = b$, then $d(a,b) = d(\pi(a),b)$.
By the first fact, if $\pi$ fixes $u$, then it also fixes $v$.
We finish the proof with the second fact by observing that if $j \neq j'$ then $d(u,x_{j,i}) \neq d(u,x_{j',i'})$ or $d(v,x_{j,i}) \neq d(v,x_{j',i'})$.
\end{proof}

Let us constructed the smaller weighted graph $(G_{\ell,k})_u$ as in Theorem \ref{reduction theorem}.
Let $V_* = \{v_*, u_*, x_{1,*}, x_{2,*}, \ldots, x_{\ell-1,*}\}$ and define $\phi$ by $\phi(v) = v_*$, $\phi(u) = u_*$, and $\phi(x_{j,i}) = x_{j,*}$.
There are edges $u_*v_*$, $u_*x_{1,*}$, $x_{\ell-1,*}v_*$, and $x_{j,*}x_{j+1,*}$ for $1 \leq j \leq \ell-2$.
The vertex weights are $w_V(x_{j,*}) = k$, $w_V(v_*) = 1$.
All edges have weight $k$ except $iv_*$, which has weight $1$.

\begin{theorem}\label{spectral paths for weighted cycle}
Fix $\ell$.  
For any $i$ and $k$ sufficiently large, the spectral path from $x_{\ell-1,i}$ to $u$ has length $\ell-1$, while the distance between those vertices is $2$.
\end{theorem}
\begin{proof}
We consider the weighted graph $(G_{\ell,k})_u$ where each edge weight and vertex weight is multiplied by $k^{-1}$.
The function optimizing (\ref{f_i as weighted optimization problem}) does not change if we scale the weights of the vertices or edges by a uniform scalar.
Let $G_{(\ell)}$ be the resulting weighted graph when $k \rightarrow \infty$ and all edges with zero weight are deleted.
Specifically, $G_{(\ell)}$ is a path with endpoints $u_*$ and $v_*$ where every edge has weight $1$.
Each vertex has weight $1$ except $v_*$, which has weight $0$.
By Theorem \ref{spectral paths are well-defined} there is a spectral path from $x_{\ell-1,*}$ to $u_*$ in $G_{(\ell)}$.
There is only one path from $x_{\ell-1,*}$ to $u_*$, and it has length $\ell-1$.
By Corollary \ref{limiting works}, for sufficiently large $k$, this is also the spectral path in $(G_{\ell,k})_u$, which can be lifted to the spectral path in $G_{\ell,k}$ that starts at $x_{\ell-1,i}$.
\end{proof}

\subsection{Double Broom With Extra Structure}\label{const 2 subsec}

We construct an unweighted graph family $H_{\ell,k,t}$, which is indexed by positive integers $\ell,k>2$ and positive real $t$.
The graph $H_{\ell,k,t}$ contains a subgraph $B_{\ell,k}$ that will be used later, and so we describe it separately.
We start constructing $B_{\ell,k}$ with $2k$ disjoint paths: for each $1 \leq i \leq k$ create path $x_{0,i}x_{1,i}x_{2,i}\ldots x_{\ell+1,i}$ and path $y_{0,i}y_{1,i}y_{2,i}y_{3,i}y_{4,i}y_{5,i}$.
The vertices $x_{0,1},y_{0,1},x_{0,2},y_{0,2},\ldots ,x_{0,k},y_{0,k}$ are identified to a single vertex $u$, and the vertices $x_{\ell+1,1},y_{5,1},x_{\ell+1,2},y_{5,2},\ldots ,x_{\ell+1,k},y_{5,k}$ are identified to a single vertex $v$.
The vertices $y_{2,1},y_{2,2},\ldots,y_{2,k}$ are identified into a single vertex $y_{2,*}$, and the vertices $y_{3,1},y_{3,2},\ldots,y_{3,k}$ are identified into a single vertex $y_{3,*}$ (implicitly, the $k$ edges $y_{2,i}y_{3,i}$ become a single edge $y_{2,*}y_{3,*}$).
We illustrate $B_{6,3}$ in Figure \ref{construction 2 image}.
In this construction, we call vertices $u$ and $v$ the ``connectors'' of the subgraph.

\begin{figure}[htbp]
\begin{center}
\includegraphics[width=3in]{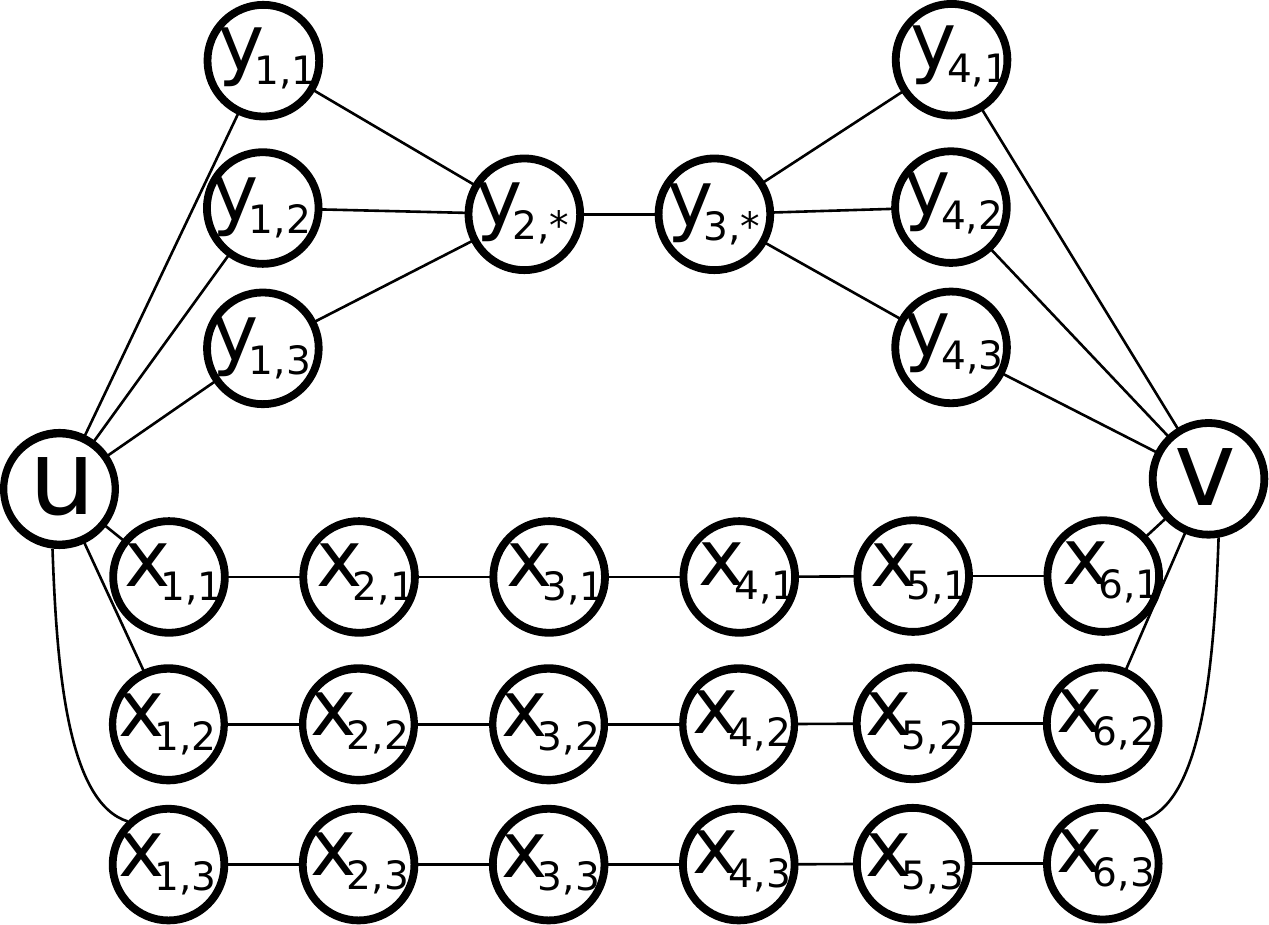}
\caption{This is graph $B_{6,3}$.}\label{construction 2 image}
\end{center}
\end{figure}

Let $T = \lfloor tk \rfloor$.
We construct $H_{\ell,k,t}$ from $B_{\ell,k}$ by adding $T$ pendant vertices to each of $u$ and $v$.
We illustrate $H_{6,3,5/3}$ in Figure \ref{construction 2 image 2}.
Let $u_1, \ldots, u_{T}$ be the pendant vertices adjacent to $u$, and let $v_1, \ldots, v_T$ be the pendant vertices adjacent to $v$.
For any $i,j$, the path from $u_i$ to $v_j$ is at most $7$.
We will show that for fixed $\ell$, the symmetric spectral path between $u_i$ and $v_j$ is $\ell+3$ when $k$ is sufficiently large.
If we chose a vertex pair at random from $H_{\ell,k,t}$, the probability that we would choose $\{u_i,v_j\}$ for some $i,j$ is $\frac{2T^2}{(2T + 4 + k(\ell + 2))^2}$, which converges to $\frac{1}{2}$ as $t$ grows (here, we fix both $\ell$ and $k$).

\begin{figure}[htbp]
\begin{center}
\includegraphics[width=3.5in]{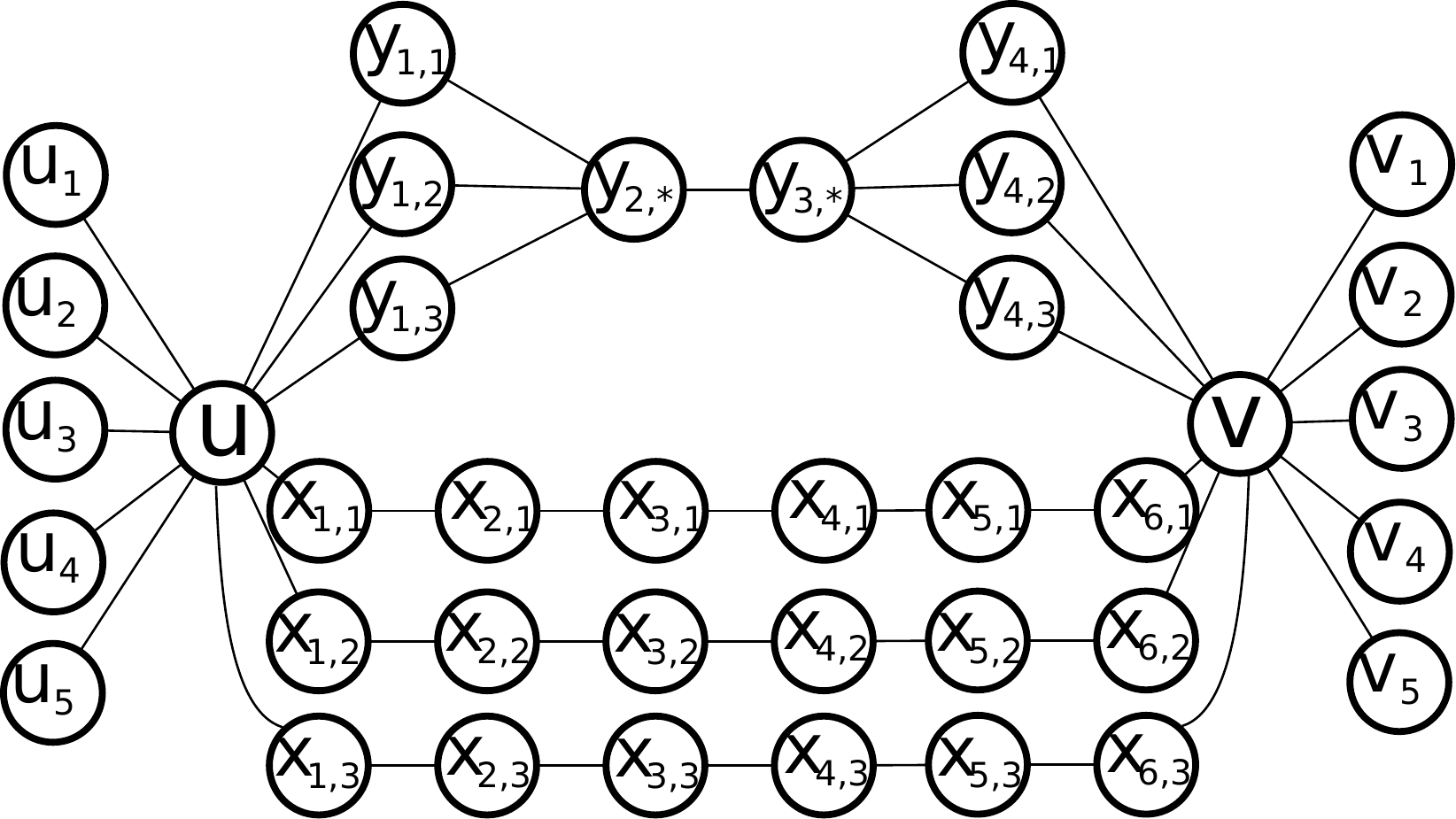}
\caption{This is graph $H_{6,3,\frac{5}{3}}$.}\label{construction 2 image 2}
\end{center}
\end{figure}

In the following, we consider special vertex $u_1$ and study the spectral path from $v_1$ to $u_1$.
Observe that by symmetry, this will describe the spectral path from $u_i$ to $v_j$ and the spectral path from $v_j$ to $u_i$ for all $i,j$.
Because every path that ends at $u_1$ has $u$ as the penultimate step, it suffices to establish the spectral path from $v_1$ to $u$.
Hence, we will deal with special vertex $u$ instead of $u_1$.
We omit the proof to the following statement, as it follows---in a longer, more tedious way---the same argument as for Proposition \ref{describing equivalence classes is weighted cycle}.

\begin{prop}
Consider $H_{\ell,k,t}$ for $k > 2$ and $\ell > 6$.
The equivalence classes of $\sim_{u}$ are $\{u_i\}_{i}$, $\{u\}$, $\{x_{1,i}\}_i, \ldots, \{x_{\ell,i}\}_i$, $\{y_{1,i}\}_i$, $\{y_{2,*}\}$, $\{y_{3,*}\}$, $\{y_{4,i}\}_i$, $\{v\}$, $\{v_i\}_i$.
\end{prop}

Let us constructed the smaller weighted graph $(H_{\ell,k,t})_{u}$ as in Theorem \ref{reduction theorem}.
Let 
$$V_* = \{v_*, v', u_*, u',  x_{1}', \ldots, x_{\ell}',  y_{1}', \ldots, y_{4}'\}.$$
We define $\phi$ and $w_V$ by 
\begin{itemize}
	\item $\phi(v) = v_*$, $\phi(u) = u_*$, $\phi(y_{2,*}) = y_2'$, $\phi(y_{3,*}) = y_3'$, each with weight $1$;
	\item $\phi(v_i) = v'$ with weight $T$;
	\item $\phi(u_i) = u'$ with weight $T$;
	\item $\phi(x_{j,i}) = x_{j}'$, each with weight $k$; and
	\item $\phi(y_{j,i}) = y_{j}'$ for $j \in \{1,4\}$, each with weight $k$.
\end{itemize}
All edges have weight $k$ except for $w_E(y_2',y_3') = 1$ and $w_E(v_*,v') = w_E(u_*, u') = T$.

\begin{theorem}\label{spectral paths for double broom}
Fix $\ell>6$ and $t'$, and consider the graph $H_{\ell,k,t}$.
For $k$ sufficiently large, the spectral path from $v_1$ to $u_1$ has length $\ell+3$, while the distance between those vertices is $7$.
\end{theorem}
\begin{proof}
Recall that the spectral path from $v_1$ to $u_1$ is the spectral path from $v_1$ to $u$ with the edge $uu_1$ appended.
Thus, it suffices to prove that the spectral path from $v_1$ to $u$ has length $\ell_2$.

We consider the weighted graph $(H_{\ell,k,t})_{u}$ where each edge weight and vertex weight is multiplied by $k^{-1}$.
Let $H_{(\ell,t)}$ be the resulting weighted graph when $k \rightarrow \infty$ and the edge $y_2'y_3'$ (the unique edge with weight zero) is deleted.
Observe that $w_V(v') = t > 0$, so by Theorem \ref{spectral paths are well-defined} there exists a spectral path from $v'$ to $u_*$ in $H_{(\ell,t)}$.
The resulting graph $H_{(\ell,t)}$ is a tree, so the spectral path from $v'$ to $u_*$ must be the unique path from $v'$ to $u_*$, which is $v'v_*x_\ell'x_{\ell-1}'\cdots x_1'u_*$ and has length $\ell+2$.
By Corollary \ref{limiting works}, for sufficiently large $k$, this is also the spectral path in $(H_{\ell,k,t})_{u}$, which can be lifted to the spectral path in $H_{\ell,k,t}$ that begins at $v_1$.
\end{proof}

\subsection{A Path for a Block Structure}\label{const 3 subsec}

We construct an unweighted graph family $J_{\ell,k,d}$, which is indexed by even positive integer $\ell > 5$ and positive integers $k>2,d>\ell$.
For fixed $\ell,k,d$ we construct $J_{\ell,k,d}$ by taking a copy of $B_{\ell,k}$ with connectors $u_i,v_i$ for each $1 \leq i \leq d$.
We will name copy $i$ of $B_{\ell,k}$ as $B_i$.
We then identify $v_i$ with $u_{i+1}$ for each $1 \leq i < d$.
We then add $k$ pendant vertices $x_1, \ldots, x_k$ attached at $u_1$ and $k$ more pendant vertices $y_1, \ldots, y_k$ attached at $v_d$.

For $d > \ell > 6$, the diameter of $J_{\ell,k,d}$ is $5d + \ell - 4$, which is achieved by taking the distance from the copy of $x_{\ell/2-2,1}$ in $B_1$ to the copy of $x_{\ell/2+3,1}$ in $B_d$.
Using arguments similar to those in Section \ref{const 2 subsec}, it is clear that the length of the symmetric spectral path from $x_1$ to $y_1$ is $(\ell + 1) d + 2$.
Thus, spectral paths can have length that is an unbounded multiple of the diameter of the overall space.

\section{Conditions for Bounded Length Spectral Paths}\label{pos result sec}

In Section \ref{spread subsec} we will prove Theorem \ref{positive result}, which is our positive result.
In Section \ref{Ricci curve sec} we present two questions whose answer could lead to a positive answer to Question \ref{motivating question}.

\subsection{Spread}\label{spread subsec}
For simplicity, we will assume the special vertex $i$ is fixed during this section.

We consider an optimization problem that can informally be interpreted as replacing an $\ell_2$ norm in (\ref{original laplace as optimization}) with an $\ell_\infty$ norm:
\begin{equation}\label{spread optimization}
 s_* := \min_{\|g\|_2 = 1, 0 = \sum_i g(i)} \max_{ij \in E} |g(i) - g(j)| . 
\end{equation}
The \emph{spread} of a graph is $s_*^{-1}$.
We consider a related optimization problem; it is derived from (\ref{spread optimization}) using  the same transformation that turned (\ref{original laplace as optimization}) into (\ref{f_i as optimization problem}):
\begin{equation}\label{modified spread eqn}
 \min_{\|g\|_2 = 1, 0 = g(i)} \max_{ij \in E} |g(i) - g(j)| . 
\end{equation}
Let $\tilde{f}$ be a function that solves optimization problem (\ref{modified spread eqn}), and let $s = \max_{ij \in E} |\tilde{f}(i) - \tilde{f}(j)|$.
By replacing $\tilde{f}$ with $-\tilde{f}$ when necessary, we assume that there exists a $j$ with $\tilde{f}(j) > 0$.

\begin{prop}\label{main tool for spread}
For each $j \neq i$ we have $\tilde{f}(j) > 0$.
Moreover, there exists a $j' \in N(j)$ such that $\tilde{f}(j') = \tilde{f}(j) - s$.
\end{prop}
\begin{proof}
For a nonzero function $g$ with $g(i) = 0$, we define function
$$ c'(g) = \frac{\max_{ij \in E} |g(i) - g(j)|}{\|g\|_2}. $$
The function $\tilde{f}$ can be taken as $g'/\|g\|_2$ for a function $g'$ that minimizes $c'$.
The proof to the first statement of the proposition then follows the same argument as Theorems \ref{core weighted result part 1}.
The proof to the second statement then follows the same argument as Theorem \ref{core weighted result part 2}.
\end{proof}

We construct the tree $T_i^{(c)}$ by replacing $f_i$ with $\tilde{f}$ in the definition of $T_i$ in the Introduction.
By Proposition \ref{main tool for spread}, $T_i^{(c)}$ is well-defined.
We can define a \emph{spread-path from $j$ to $i$} as the path from $j$ to $i$ in $T_i^{(c)}$.

\begin{theorem}\label{positive result}
A spread-path from $j$ to $i$ is a shortest path from $j$ to $i$.
\end{theorem}
\begin{proof}
By construction, $\tilde{f}$ monotonically decreases along any spread-path to $i$.
We will show that it decreases at a linear rate.
Let $L:V_i \rightarrow \mathbb{R}$ be a function defined by $L(j') = \tilde{f}(j') - \min_{w \in N(j')} \tilde{f}(w)$.
That is, $L$ represents how much $\tilde{f}$ decreases across the next edge of the spread-path to $i$.

We claim that $L \equiv s$.
Let $j'$ be an arbitrary vertex other than the special vertex $i$.
By Proposition \ref{main tool for spread}, $L(j') \geq s$.
By the definition of $s$, we have $L(j') \leq s$.
Thus, the claim is proved.

The above claim implies that for each $j$ we have that $s^{-1}\tilde{f}(j)$ is an integer and equals the length of the spread-path from $j$ to $i$.
The theorem follows from proving that every path from $j$ to $i$ has length at least $s^{-1}\tilde{f}(j)$
Consider a path $j = x_0x_1\cdots x_\ell=i$.
By definition of $L$, we have for $i < \ell$ that $\tilde{f}(x_{i+1}) \geq \tilde{f}(x_i) - L(x_i) = \tilde{f}(x_i) - s$.
Therefore $\ell \geq s^{-1}(\tilde{f}(x_0) - \tilde{f}(x_{\ell})) = s^{-1}(\tilde{f}(j)-0)$.
\end{proof}

\subsection{Graphs with bounded curvature}\label{Ricci curve sec}
The central claim to the proof of Theorem \ref{positive result} is that $L$ is a constant function.
This is stronger than necessary; a positive result towards Question \ref{motivating question} would only require a uniform lower and upper bound on $L$.
We are not the first to be interested in a uniform upper bound: Chung and Yau \cite{CY2017} stated an interest in bounds on ``stretches $|f(x) - f(y)|$ of any given edge'' $xy$ for ``combinatorial eigenfunction $f$.''

To avoid duplication of notation, we will use $f_D$ to denote our ``combinatorial eigenfunction.''
In the contexts of Chung and Yau's studies, $f_D$ is typically the Fiedler right-eigenvector of the \emph{random-walk Laplacian}, which is $D^{-1}L$, although they also consider other right-eigenvectors of that matrix.
Equivalently, $f_D$ is the vector that solves the optimization problem $\min_{x \neq 0}\frac{x^T L x}{x^T D x}$.
Steinerberger (\cite{S}, Section 2.5) does consider the random-walk Laplacian, and explains the focus on the Laplacian to be due to empirical performance.
Let $\lambda$ be the eigenvalue associated to $f_D$.
Let $L_D(x) = f_D(x) - \min_{y \in N(x)} f_D(x)$ when $f_D(x) \geq 0$ and $L_D(x) = -f_D(x) + \min_{y \in N(x)} f_D(x)$ otherwise; so that $L_D$ is the analogue of $L$ for the combinatorial eigenfunction.

Let $\Delta = \max_j d(j)$ denote the maximum degree.
Let $\alpha = \max_j |f_D(j)|$.
We assume that by normalization we have $\|f_D\|_2 = 1$, which implies that $\alpha \in [0,1]$.

Chung and Yau \cite{CY2017} established an upper bound on $L_D$ when the graph satisfies graph curvature conditions. 
There are many different variations of graph curvature; we survey eight variations of positive curvature on Page 545 of \cite{Y_Pos} and five variations of negative curvature are discussed at various places in \cite{Y_Neg}; neither survey includes either of the two curvature variations in \cite{CY2017}.
The central definition of curvature in \cite{CY2017} is simply called ``curvature;'' let $\kappa$ denote the curvature of $G$ when it is well-defined.
A connection between $\kappa$ and other forms of curvature can be found in Section 3.3 of \cite{CY2017}.
Chung and Yau \cite{CY2017} show that $L_D \leq  \alpha \sqrt{(8\Delta +4 \kappa) \lambda}$.
This is a generalization of an earlier result \cite{CY1994} by the same authors for a class of graphs that satisfy $\kappa = 0$, which includes Cayley graphs of abelian groups.

On the other hand, a direct application of the definition of an eigenvector of $D^{-1}L$ gives a (non-uniform) lower bound $L_D(x) \geq \lambda f_D(x)$.

Motivated by finding a second positive answer to Question \ref{motivating question}, we pose the following two questions.
Can the work of Chung and Yau be adapted to the presence of a special vertex?
Is there a uniform lower bound on $L_D$ that incorporates the diameter of the overall space (possibly under an assumption of curvature)?
Perhaps a weaker result would be a uniform lower bound that incorporates $n$?



\begin{thebibliography}{1}

	\bibitem{AvDF} S. Akbari, E.R. van Dam, and M.H. Fakharan, ``Trees with a large Laplacian eigenvalue multiplicity.'' Linear Algebra and its Applications 586 (2020) 262--273.

	\bibitem{CY1994} F. Chung and S.-T. Yau, ``A Harnack inequality for homogeneous graphs and subgraphs,'' Comm. Analysis and Geometry 2 (1994) 627--640.

	\bibitem{CY2017} F. Chung and S.-T. Yau, ``A strong Harnack inequality for graphs,'' Comm. Analysis and Geometry 25(3)  (2017) 557--588.

	\bibitem{S}  S. Steinerberger \url{https://arxiv.org/abs/2004.01163} .

	\bibitem{Y_Pos} M. Yancey, ``Positively Curved Graphs.'' J. Graph Theory 94(4) (2020) 539--578.

	\bibitem{Y_Neg} M. Yancey, ``Negatively Curved Graphs.'' \url{https://arxiv.org/abs/1512.01281v2}.

\end{thebibliography}
\end{document}